\documentclass[10pt, reqno]{amsart}
\usepackage{amscd,graphicx,epsfig}
\usepackage{amssymb}
\usepackage{mathrsfs}
\usepackage[colorlinks=true, pdfstartview=FitV, linkcolor=blue, citecolor=blue, urlcolor=blue]{hyperref}
\usepackage[all]{xy}
\usepackage{enumerate}
\usepackage{setspace}

\title[Automorphisms of $\CC^3$ Commuting with a $\CC^+$-Action]
{Automorphisms of $\CC^3$ Commuting with a 
$\CC^+$-Action}
\author{Immanuel Stampfli}

\address{Jacobs University Bremen gGmbH, School of Engineering and Science, Department of Mathematics, Campus Ring 1, 28759 Bremen, Germany}
\email{immanuel.e.stampfli@gmail.com}
\thanks{The author was supported by the
Swiss National Science Foundation (Schweizerischer Nationalfonds),
grants ``Automorphisms of Affine $n$-Space" 137679
and ``Automorphisms of affine algebraic varieties" 148627}

\theoremstyle{plain}
\newtheorem{thm}{Theorem}[subsection]
\newtheorem{prop}[thm]{Proposition}
\newtheorem{lem}[thm]{Lemma}
\newtheorem{cor}[thm]{Corollary}

\theoremstyle{definition}
\newtheorem{defn}{Definition}[subsection]

\theoremstyle{remark}
\newtheorem{rem}[defn]{Remark}

\newcommand{\name}[1]{#1}

\renewcommand{\AA}{{\mathbb C}}
\newcommand{\PP}{{\mathbb P}}
\newcommand{\AAone}{{\mathbb C}}
\DeclareMathOperator{\GGa}{\CC^+}

\DeclareMathOperator{\Cent}{Cent}
\DeclareMathOperator{\Exp}{Exp}
\DeclareMathOperator{\im}{im}

\DeclareMathOperator{\rank}{rank}
\DeclareMathOperator{\ZZ}{\mathbb Z}
\DeclareMathOperator{\NN}{\mathbb N}
\DeclareMathOperator{\CC}{{\mathbb C}}

\newcommand{\g}{\mathbf g}
\newcommand{\h}{\mathbf h}
\newcommand{\bu}{\mathbf u}

\newcommand{\f}{\mathbf f}
\newcommand{\e}{\mathbf e}
\newcommand{\br}{\mathbf r}

\newcommand{\bt}{\mathbf t}
\newcommand{\w}{\mathbf w}

\newcommand{\bv}{\mathbf v}
\newcommand{\bid}{\mathbf{id}}

\renewcommand{\div}{\mathrm{div}}
\DeclareMathOperator{\pr}{pr}

\newcommand{\ud}{\,\mathrm{d}}

\DeclareMathOperator{\red}{\textrm{red}}
\newcommand{\OO}{\mathcal O}

\DeclareMathOperator{\Spec}{Spec}

\DeclareMathOperator{\Aut}{Aut}
\DeclareMathOperator{\End}{End}
\DeclareMathOperator{\Iner}{Iner}
\DeclareMathOperator{\id}{id}
\DeclareMathOperator{\Span}{span}
\DeclareMathOperator{\ad}{ad}

\newcommand{\aquot}{/ \! \! /}

\begin{document}

\begin{abstract}
	Let $\rho$ be an algebraic action of the additive group $\CC^+$
	on the three-dimensional affine space $\CC^3$.
	We describe the group $\Cent(\rho)$ of polynomial automorphisms of $\CC^3$
	that commute with $\rho$. A particular emphasis lies in the
	description of the automorphisms in $\Cent(\rho)$ coming from
	algebraic $\CC^+$-actions. As an application we prove
	that the automorphisms in $\Cent(\rho)$ that are
	the identity on the algebraic quotient of $\rho$ form a characteristic
	subgroup of $\Cent(\rho)$.
\end{abstract}

\maketitle


\section{\texorpdfstring{Introduction}{Introduction}}
Let $X$ be an affine algebraic variety. 
A classification of the algebraic $\CC^+$-actions on 
$X$ up to conjugacy in the automorphism group $\Aut(X)$ is only known for a few varieties $X$. For example, when $X = \CC^2$ we have a classification: 
every $\CC^+$-action is a modified translation up to conjugacy, 
i.e. an action of the form $t \cdot (x, y) = (x + td(y), y)$ for a suitable polynomial $d$
(see \cite{Re1968Operations-du-grou}).
In contrast to the two-dimensional case, there is no classification known for the 
$\CC^+$-actions on $\CC^3$. As a first step towards a classification, we
study the centralizer of a $\CC^+$-action in $\Aut(\CC^3)$, i.e. the 
group of automorphisms that commute with the $\CC^+$-action.


An automorphism $\bu \in \Aut(\CC^n)$ is called 
\emph{unipotent} if there exists on $\CC^n$ an algebraic $\CC^+$-action
$\rho \colon \CC^+ \times \CC^n \to \CC^n$ such that 
\[
	\bu(x_1, \ldots, x_n) = \rho(1, x_1, \ldots, x_n) \quad \textrm{for all 
	$(x_1, \ldots, x_n) \in \CC^n$} \, .
\]
In fact, the $\CC^+$-action $\rho$ is uniquely determined 
by the unipotent automorphism 
$\rho_1 = ( (x_1, \ldots, x_n) \mapsto \rho(1, x_1, \ldots, x_n) )$.
Thus $\rho \mapsto \rho_1$ is a bijection between algebraic $\CC^+$-actions
on $\CC^n$ and unipotent automorphisms of $\CC^n$. Moreover,
the centralizer of $\rho$ is the same as the centralizer of $\rho_1$ and thus 
we are reduced to the study of the centralizer of a unipotent automorphism.

For us it will be crucial
that there is a different characterization of unipotent automorphisms.
Namely, there is a bijective correspondence of unipotent automorphisms
of $\CC^n$ and locally nilpotent derivations of the 
polynomial ring $\OO(\CC^n) = \CC[x_1, \ldots, x_n]$, 
given by the exponential $\Exp$
\begin{eqnarray*}
	\{ \, \textrm{locally nilpotent derivations of $\OO(\CC^n)$} \, \} 
	& \hspace{-5pt} \stackrel{1:1}{\longleftrightarrow} \hspace{-5pt} & 
	\{ \, \textrm{unipotent automorphisms of $\CC^n$} \, \} \\
	D & \hspace{-5pt}  \longmapsto \hspace{-5pt}  & \Exp(D)
\end{eqnarray*}
compare \cite[sec. 1.5]{Fr2006Algebraic-theory-o}.

In dimension $n=2$, \name{Shmuel Friedland} and \name{John Milnor} proved
that every automorphism of $\AA^2$ is conjugate 
to a composition of so-called generalized \name{H\'enon} maps 
or to a triangular automorphism, i.e. an automorphism
of the form $(x, y) \mapsto (ax + b(y), cy + e)$ where 
$a, c \in \CC^\ast$, $b \in \CC[y]$ and $e \in \CC$ 
(cf. \cite[Theorem~2.6]{FrMi1989Dynamical-properti}). In the first case,
\name{St\'ephane Lamy} showed that the centralizer of such an automorphism
is isomorphic to a semi-direct product of $\ZZ$ with a finite cyclic group $\ZZ_q$
(cf. \cite[Proposition~4.8]{La2001Lalternative-de-Ti}). In the second case,
assuming in addition that the automorphism is unipotent, it has
the form $\bu(x, y) = (x + d(y), y)$ for suitable coordinates $(x, y)$ of 
$\CC^2$. One can see 
that the centralizer $\Cent(\bu)$ consists of triangular automorphisms and that it
fits into the following split short exact sequence 
\begin{equation}
	\label{2dim.eq}
	1 \to \{ \, (x, y) \mapsto (x + b(y), y) \ | \ b \in \CC[y] \, \}
	\hookrightarrow \Cent(\bu) 
	\stackrel{p}{\twoheadrightarrow} \Aut(\AAone, \div(d)) \to 1
\end{equation}
where $\Aut(\AAone, \div(d))$ denotes the automorphisms of $\AAone$
that preserve the principal divisor $\div(d) \subseteq \CC$
and $p$ sends $((x, y) \mapsto (ax +b(y), cy + e))$ to
$(y \mapsto cy + e)$.

In dimension $n=3$, \name{Cinzia Bisi} proved the following for a regular
automorphism $\f$ (i.e. an automorphism such that the indeterminacy sets of
the map and its inverse seen as birational maps of $\PP^3$ do not intersect):
If $\g$ is any automorphism that commutes with $\f$, then
$\g^m = \f^k$ for certain integers $k, m$
(cf. \cite[Main~Theorem~1.1]{Bi2008On-commuting-polyn}). 
As a counterpart to the regular automorphisms, one can regard
the unipotent automorphisms (a regular automorphism satisfies 
$\deg(\f^k) = \deg(\f)^k$ for all integers $k \geq 0$ and thus can not be unipotent). 
The work of \name{David Finston} and \name{Sebastian Walcher} 
\cite{FiWa1997Centralizers-of-lo} can be seen 
as a first step in the study of the centralizer of a unipotent automorphism.
They explore the centralizer of a triangulable (locally nilpotent) derivation inside 
the algebra of all derivations of $\OO(\CC^3)$.

In this article we study the structure of the centralizer of a unipotent automorphism 
$\bu \in \Aut(\CC^3)$ with the aid of an exact sequence similar to~\eqref{2dim.eq}
and give some applications of this study.

\section{\texorpdfstring{Statement of the main results}
	    {Statement of the main results}}

Let $\bu \in \Aut(\CC^3)$ be a unipotent automorphism $\neq \bid$
and let $D$ be the locally nilpotent derivation of $\OO(\CC^3)$
such that $\bu = \Exp(D)$.
Denote by $\OO(\CC^3)^\bu$ the functions that are invariant under the 
$\CC^+$-action induced by $\bu$. They satisfy 
$\OO(\CC^3)^\bu = \ker D \subseteq \OO(\CC^3)$.
Denote by $\AA^3 \aquot \bu$ the algebraic quotient of $\CC^3$ by the 
$\CC^+$-action induced by $\bu$, i.e. 
$\CC^3 \aquot \bu = \Spec(\OO(\CC^3)^\bu)$.
By \name{Miyanishi}'s Theorem (cf. \cite[Theorem~5.1]{Fr2006Algebraic-theory-o}),
this algebraic quotient is isomorphic to $\AA^2$. 

We call the ideal $\im D \cap \ker D$ of $\ker D$ 
the \emph{plinth ideal} of $D$.
By \cite[Theorem~1]{DaKa2009A-note-on-locally-} this ideal
is principal. We fix some generator $a$ and call the principal divisor
\[
	\Gamma = \div(a) \subseteq \AA^3 \aquot \bu = \CC^2
\]
the \emph{plinth divisor} of $\bu$ (respectively of $D$). 
We get a homomorphism of the centralizer
$\Cent(\bu)$ into the group $\Aut(\CC^3 \aquot \bu, \Gamma)$ 
of automorphisms of the algebraic quotient 
$\CC^3 \aquot \bu$ that preserve the plinth divisor $\Gamma$
\[
	p \colon \Cent(\bu) \to \Aut(\CC^3 \aquot \bu, \Gamma)
\]
(see sec.~\ref{RelativeAuto.sec} 
for an exact definition of $\Aut(\CC^3 \aquot \bu, \Gamma)$).
In contrast to the two-dimen\-sional case (see \eqref{2dim.eq}), the homomorphism
$p$ in the three-dimensional case is in general not surjective 
(see \cite[Proposition~1]{St2012A-note-on-Automorp}). 

If $f \in \ker D = \OO(\CC^3)^\bu$, then $f D$ is a locally 
nilpotent derivation of $\OO(\CC^3)$. 
We call $\Exp(fD)$ a \emph{modification} of $\bu = \Exp(D)$ and write
\[
	f \cdot \bu = \Exp(f D) \, .
\]
Note that the composition of two modifications is given by 
$(f \cdot \bu) \circ (f' \cdot \bu) = (f+f') \cdot \bu$.
Thus the modifications of $\bu$ form a group. We denote this group by
$\OO(\CC^3)^\bu \cdot \bu$.

For the sake of simplicity, we 
formulate in this section several results only for irreducible $\bu$
(see subsec.~\ref{BasicPropertiesOfAUnipotAuto.subsec} for the definition 
of ``irreducible"). 
However, in the course of this article we state and prove 
the results for general unipotent $\bu$.

\begin{thm}[cf. Theorem~\ref{firstfam.prop}]
	If $\bu$ is irreducible, then the kernel of $p$ is the subgroup
	of modifications of $\bu$:
	\[
		\ker p = \OO(\CC^3)^\bu \cdot \bu \, .
	\]
\end{thm}

The description of $\Cent(\bu)$ is special in the case when $\bu$ is a 
\emph{translation}, i.e. $\bu(x, y, z) = (x+1, y, z)$ for suitable coordinates $(x, y, z)$.
Note that $\bu$ is a translation if and only if its plinth divisor $\Gamma$ 
is empty.

\begin{prop}[cf. Proposition~\ref{modtranslation.prop}]
	The unipotent automorphism $\bu$ is a translation, if and only if 
	there is a split short exact sequence
	\[
		1 \to \OO(\AA^3)^{\bu} \cdot \bu 
		\hookrightarrow \Cent(\bu) \stackrel{p}{\to}
		\Aut(\AA^3 \aquot \bu) \to 1 \, .
	\]
\end{prop}

In order to formulate the next results, we recall briefly some notion and facts 
of the theory of ind-groups 
(see \cite[ch. IV]{Ku2002Kac-Moody-groups-t} for an introduction). 
A group $G$ is called an \emph{ind-group} if it is endowed with a filtration by
affine varieties $G_1 \subseteq G_2 \subseteq \ldots \,$, each one closed in the
next, such that $G = \bigcup_{i=1}^{\infty} G_i$ and such that the map
$G \times G \to G$, $(x, y) \mapsto x \cdot y^{-1}$ 
is a morphism of ind-varieties. 
We endow $G$ with the following topology: 
$X \subseteq G$ is closed if and only if $X \cap G_i$ is closed in 
$G_i$ for each $i$.
We call a subgroup $H$ of $G$ \emph{algebraic}, if it is a closed subset of
some $G_i$. An element $g \in G$ is called \emph{algebraic}, if it is contained in
some algebraic subgroup of $G$.

For example, $\Aut(\AA^n)$
is an ind-group with the filtration 
$\Aut(\AA^n)_1 \subseteq \Aut(\AA^n)_2 \subseteq \ldots \,$ 
where $\Aut(\AA^n)_i$ is the set of all automorphisms of degree $\leq i$ 
(see \cite{BaCoWr1982The-Jacobian-conje}). Then, $\Aut(\CC^2, \Gamma)$
is a closed subgroup of $\Aut(\CC^2)$ (cf. Proposition~\ref{curves.prop}). 
Examples of closed subgroups of $\Aut(\CC^3)$ are $\Cent(\bu)$ and
$\ker p$ (cf. Remark~\ref{indhomo.rem}). 
As a last example, the group of modifications
$\OO(\CC^3)^\bu \cdot \bu$ is closed in $\Aut(\CC^3)$. Moreover, 
the homomorphism
$(\OO(\CC^3)^\bu, +) \to \OO(\CC^3)^\bu \cdot \bu$, $f \mapsto f \cdot \bu$
is an isomorphism of ind-groups. 
In particular, the closure of $\langle \bu \rangle$ in 
$\Aut(\CC^3)$ is the group $\CC \cdot \bu$ and thus $\bu$ is an algebraic element 
of $\Aut(\CC^3)$.

\begin{thm}[cf. Theorem~\ref{main.thm1},
\ref{main.thm2} and Corollary~\ref{mainthm2.cor}] 
	\label{unipotgeneral.mtm}
	Assume that $\bu$ is not a translation. Then
	\begin{enumerate}[i)]
		\item All elements in $\Cent(\bu)$ are algebraic.
		\item The set of unipotent elements $\Cent(\bu)_u \subseteq \Cent(\bu)$
		is a closed normal subgroup. 
		\item There exists an algebraic
		subgroup $R \subseteq \Cent(\bu)$ consisting only 
		of semi-simple elements such that 
		$\Cent(\bu) \simeq \Cent(\bu)_u \rtimes R$ 
		(as ind-groups). In particular, the connected component of the
		identity in $R$ is a torus.
	\end{enumerate}
\end{thm}

The next result describes the group of unipotent elements $\Cent(\bu)_u$.
For this, we need the following terminology. 
We call the divisor $\Gamma = \sum_i n_i \Gamma_i$ in $\CC^2$ 
a \emph{fence}, if
$\Gamma_i \simeq \AAone$ for all $i$ and the $\Gamma_i$ are pairwise disjoint.

\begin{thm}[cf. Proposition~\ref{properties.prop} and 
Theorem~\ref{main.thm2}] Assume $\bu$ is irreducible and not a translation.
If the plinth divisor $\Gamma$ of $\bu$
\begin{enumerate}[i)]
	\item is not a fence, then $\Cent(\bu)_u = \ker p$.
	\item is a fence, then 
	$\Cent(\bu)_u \simeq \ker p \rtimes \Iner(\CC^3 \aquot \bu, \Gamma)_u$ 
	(as ind-groups) where
	$\Iner(\CC^3 \aquot \bu, \Gamma)_u$
	is the subgroup of unipotent automorphisms in 
	$\Aut(\CC^3 \aquot \bu, \Gamma)$ that are the identity on 
	$\Gamma$ (cf. sec.~\ref{RelativeAuto.sec} for an exact definition).
	Moreover, there exists an irreducible $\e \in \Cent(\bu)_u$ such
	that $p$ induces an isomorphism of ind-groups
	 $\OO(\AA^3)^{\langle \bu, \e \rangle} \cdot \e \simeq 
	 \Iner(\CC^3 \aquot \bu, \Gamma)_u$
	where $\OO(\CC^3)^{\langle \bu, \e \rangle}$ 
	is the subring of $\e$-invariant polynomials inside $\OO(\CC^3)^\bu$.
\end{enumerate}
\end{thm}

Let us give some explanation of the last result.
If $\Gamma$ is not a fence, then the underlying variety 
cannot be a union of orbits of a non-trivial $\GGa$-action on $\AA^2$. 
Hence, all unipotent automorphisms
of $\Cent(\bu)$ induce the identity on the algebraic quotient
and thus the first part of the result follows.
So let us assume that $\Gamma$ is a non-empty fence.
There exists a proper non-empty open subset $U \subseteq \AAone$ such that
the algebraic quotient $\pi \colon \CC^3 \to \CC^3 \aquot \bu$ fits into 
the following commutative diagram
\[
	\xymatrix@=1.5pc{
		\AA^3 \ar[d]_-{\pi} \ar@{}[r]|-{\supseteq} & 
		\AA^3 \setminus \pi^{-1}(\Gamma) \ar@{->>}[d] \ar@{}[r]|-{\simeq} &
		(U \times \AAone) \times \AAone \ar@{->>}[d]^-{\pr}  \\
		\AA^2 \ar@{}[r]|-{\supseteq} & 
		\AA^2 \setminus \Gamma \ar@{}[r]|-{\simeq} &
		U \times \AAone
	}
\]
where $\pr$ denotes the projection onto the first two factors 
(see Proposition~\ref{genAMS.prop} and 
\cite[Principle~11]{Fr2006Algebraic-theory-o}). 
There exist coordinates $(u, v, w)$ of $(U \times \AAone) \times \AAone$ such that
$\bu$ restricted to $(U \times \AAone) \times \AAone$ is given by 
$(u, v, w) \mapsto (u, v, w+1)$. The automorphism 
$(u, v, w) \mapsto (u, v+1, w)$ of $(U \times \AAone) \times \AAone$ extends to an
irreducible unipotent automorphism $\e$ on $\AA^3$ 
that commutes with $\bu$
(see subsec.~\ref{The_second_subgroup.subsec}). Thus, $\Cent(\bu)_u$ contains
$\OO(\AA^3)^{\langle \bu, \e \rangle} \cdot \e$ beside 
$\OO(\AA^3)^\bu \cdot \bu$. It is not hard to prove that $p$ induces 
an isomorphism 
$\OO(\AA^3)^{\langle \bu, \e \rangle} \cdot \e \simeq 
\Iner(\CC^3 \aquot \bu, \Gamma)_u$.
The difficulty lies in proving that 
$\Cent(\bu)_u$ is a group which is generated by 
$\OO(\AA^3)^{\langle \bu, \e \rangle} \cdot \e$ 
and $\OO(\AA^3)^\bu \cdot \bu$.

\subsection{Applications}

We present two applications of the structure theorems above.
The first one concerns abstract group automorphisms 
$\Cent(\bu) \to \Cent(\bu)$.

\begin{prop}[cf. Proposition~\ref{char.prop}]
		If $\bu$ is not a translation, then
		the subgroup
		$\ker p \subseteq \Cent(\bu)$ is
		characteristic, i.e. it is invariant  under all abstract
		group automorphisms of $\Cent(\bu)$.
\end{prop}

The second application concerns the plinth divisor $\Gamma$.
The underlying variety $\Gamma_{\red}$ 
has the following geometric description: The complement of $\Gamma_{\red}$
is the maximal open subset of $\AA^3 \aquot \bu$ 
such that the algebraic
quotient $\pi \colon \AA^3 \to \AA^3 \aquot \bu$ is a locally trivial principal
$\GGa$-bundle over it (see \cite[Proposition~5.4]{DuKr2014Invariants-and-Sep}).
So far - to the author's knowledge - 
there is no geometric description of the \emph{scheme} $\Gamma$.
But in the case when $\Gamma$ is a non-empty fence
and $\bu$ is irreducible we can give one.

\begin{prop}[cf. Proposition~\ref{applicationfence.prop}]
	If $\bu$ is irreducible and $\Gamma$ is a non-empty fence, 
	then $\Gamma$ 
	is the largest closed subscheme of $\AA^3 \aquot \bu$ 
	fixed by $\Cent(\bu)_u$.
\end{prop}

\section{\texorpdfstring{Automorphisms of $\AA^2$ that preserve a divisor}
	    {Automorphisms of A2 that preserve a divisor}}

\label{RelativeAuto.sec}

Let $\Gamma \subseteq \CC^2$ be an effective divisor.
Then there exists $a \in \OO(\CC^2)$ 
such that $\Gamma = \div(a)$ ($a$ is uniquely determined up to 
elements of $\CC^\ast$).
We denote by $\Aut(\AA^2, \Gamma)$ the subgroup of all 
$\g \in \Aut(\AA^2)$ such that the comorphism 
$\g^\ast \colon \OO(\CC^2) \to \OO(\CC^2)$
sends the principal ideal $(a)$ onto itself.
Moreover, 
we denote by $\Iner(\AA^2, \Gamma)$ the subgroup of all 
$\g \in \Aut(\AA^2, \Gamma)$ such that $\g$ is the identity on the divisor 
$\Gamma$, i.e. we have a commutative diagram
\[
	\xymatrix{
		\OO(\CC^2) \ar@{->>}[d] \ar[r]^{\g^\ast} & \OO(\CC^2) 
		\ar@{->>}[d] \\
		\OO(\CC^2) / (a)  \ar[r]^{\id} & \OO(\CC^2) / (a) \, .
	}
\]
Clearly, we get an exact sequence
\[
	1 \to \Iner(\AA^2, \Gamma) \hookrightarrow \Aut(\AA^2, \Gamma) 
	\rightarrow \Aut(\Gamma)
\]
The next result uses heavily the main 
result in \cite{BlSt2013Automorphisms-of-t}.

\begin{prop} $\textrm{ }$
	\label{curves.prop}
	Let $\Gamma$ be a non-trivial effective divisor of $\AA^2$. Then
	\begin{enumerate}[i)]

		\item The subgroup $\Aut(\AA^2, \Gamma) \subseteq \Aut(\AA^2)$ 
		is closed and all elements of $\Aut(\AA^2, \Gamma)$ are algebraic.
		
		\item The following statements are equivalent
		\begin{enumerate}[a)]
			\item $\Gamma$ is a fence
			\item $\Aut(\AA^2, \Gamma)$ contains unipotent automorphisms
				$\neq \bid$.
			\item $\Iner(\AA^2, \Gamma) \neq \{ \bid \}$
			\item $\Aut(\AA^2, \Gamma)$ is not an algebraic group
		\end{enumerate}
	\end{enumerate}
\end{prop}

For the proof of this proposition we recall some facts about $\Aut(\AA^2)$.
The next result is a direct consequence of 
\cite[Theorem~1]{BlSt2013Automorphisms-of-t}.

\begin{thm}
	\label{algcurve.thm}
	An automorphism of $\AA^2$ that preserves an algebraic curve is
	conjugate to a triangular automorphism. 
\end{thm}

In the next result we prove a slightly more general version of the
\name{Abhyankar-Moh-Suzuki}-Theorem which says that all closed
embeddings $\AAone \hookrightarrow \AA^2$ are equivalent up
to automorphisms of $\AA^2$
(see \cite{AbMo1975Embeddings-of-the-}).

\begin{prop}
	\label{genAMS.prop}
	Let $\Gamma$ be a fence in $\AA^2$ and let 
	$F \subseteq \CC$ be a closed 0-dimensional subscheme 
	such that $\Gamma \simeq F \times \CC$. Then there exists an automorphism
	of $\CC^2$ that maps $\Gamma$ onto $F \times \CC$ (scheme-theoretically).
\end{prop}

\begin{proof}
	Clearly, we can assume that $\Gamma \neq \varnothing$. 
	Moreover, we can easily reduce to the case, where $\Gamma$
	is a reduced scheme.
	Let $\Gamma_i$, $i \in I$ be the irreducible components of $\Gamma$.
	Let $i_0 \in I$ be fixed.
	By the \name{Abhyankar-Moh-Suzuki}-Theorem, there exists 
	a trivial $\AAone$-bundle $f \colon \AA^2 \twoheadrightarrow  \AAone$
	such that $\Gamma_{i_0}$ is a fiber of $f$. Now, if the restriction
	$f |_{\Gamma_i} \colon \Gamma_i \to \AAone$ is non-constant, then it
	is surjective, since $\Gamma_i \simeq \AAone$. But this implies that
	$\Gamma_i \cap \Gamma_{i_0} \neq \varnothing$, a contradiction.
	Thus every $\Gamma_i$ is a fiber of $f$. This implies the proposition.
\end{proof}

\begin{proof}[Proof of Proposition~\ref{curves.prop}] $\textrm{}$	
	\begin{enumerate}[i)]

	\item 
	Assume that $\Gamma = \div(a)$ for some non-zero 
	$a \in \OO(\CC^2) = \CC[x, y]$. 
	For $(f_1, f_2) \in \CC[x, y]^2$ we denote by $a_{ij}(f_1, f_2)$ the coefficient of
	the monomial $x^i y^j$ in the polynomial $a(f_1, f_2)$. 
	The subgroup $\Aut(\AA^2, \Gamma)$ of $\Aut(\AA^2)$ 
	is defined by the equations
	\[
		a_{ij}(f_1, f_2) a_{kl}(x, y) = a_{kl}(f_1, f_2) a_{ij}(x, y) \quad
		\textrm{for all pairs $(i, j)$, $(k, l)$} \, . 
	\]This proves the first statement.
	
	Let $\g \in \Aut(\AA^2, \Gamma)$. By Theorem~\ref{algcurve.thm}, 
	$\g$ is conjugate to a triangular automorphism and hence 
	$\g$ is algebraic. This proves the second statement.
	
	\item 
	$a) \Rightarrow b)$: This follows immediately form 
	Proposition~\ref{genAMS.prop}.
	
	$b) \Rightarrow c)$:
	Let $\g \in \Aut(\AA^2, \Gamma)$ be a unipotent 
	automorphism $\neq \bid$.
	Choose some $a \in \OO(\CC^2)$ such that $\Gamma = \div(a)$.
	As $\g$ preserves $\Gamma$, it follows that $a$ is a semi-invariant
	for the $\GGa$-action on $\AA^2$ induced by $\g$. Since $\GGa$
	has no non-trivial character, $a$ is an invariant. Hence,
	$\bid \neq a \cdot \g \in \Iner(\AA^2, \Gamma)$.
	
	$c) \Rightarrow d)$: Let $\g \in \Iner(\AA^2, \Gamma)$ with $\g \neq \bid$.
	By Theorem~\ref{algcurve.thm}, $\g$ preserves a trivial $\AAone$-bundle
	$f \colon \AA^2 \twoheadrightarrow \AAone$.
	Let $\Gamma_i$, $i \in I$ be the irreducible components of 
	the reduced scheme $\Gamma_{\red}$.
	If every $\Gamma_i$ lies in a fiber of $f$, then $\Gamma$ is a fence and
	thus $\Aut(\AA^2, \Gamma)$ is not an algebraic group. Therefore 
	we can assume that $f(\Gamma_i) \subseteq \AAone$ is dense
	for some $i$. As $\g$ is the identity on $\Gamma_i$, it follows that
	$\g$ maps each fiber on itself. Hence, there exists $\alpha \in \CC^\ast$
	and a polynomial $b(y)$ such that for each $y \in \AAone$ the restriction
	of $\g$ to the fiber $f^{-1}(y)$ is given by
	\[
		\g_y \colon \AAone \to \AAone \, , \quad
		x \mapsto \alpha x + b(y) \, .
	\]
	As $\g$ is the identity on $\Gamma_i$, it follows that $\g_y$ has a fixed point
	for all $y \in f(\Gamma_i)$.
	If $\alpha = 1$, then $\g_y$ is the identity map for all 
	$y \in f(\Gamma_i) \subseteq \AAone$.
	Since $f(\Gamma_i)$ is dense in $\AAone$ we get a 
	contradiction to the fact that
	$\g \neq \bid$.  Thus, $\alpha \neq 1$. But this implies that 
	$\g_y$ has exactly one fixed point for each $y \in \AAone$.
	Thus, $\Gamma_i = V(ax + b(y)-x) \simeq \AAone$ and it is the only
	irreducible component of $\Gamma_{\red}$. 
	Therefore, $\Gamma$ is again a fence
	and $\Aut(\AA^2, \Gamma)$ is not an algebraic group.
	
	
	$d) \Rightarrow a)$: Assume that $\Gamma$ is not a fence.
	By \cite[Theorem~1]{BlSt2013Automorphisms-of-t}, 
	$\Aut(\AA^2, \Gamma_{\red})$ is an algebraic group. Now,
	$\Aut(\AA^2, \Gamma)$ is an algebraic group as well, since
	it is a closed subgroup of
	$\Aut(\AA^2, \Gamma_{\red})$.
\end{enumerate}
	
\end{proof}
\section{Some elementary facts}
\label{BasicFacts.sec}

\subsection{Locally nilpotent derivations}
\label{BasicPropertiesofLNDs.subsec}
Let $A$ be a $\CC$-algebra 
and assume it is a unique factorization domain (UFD).
Let $D$ be a $\CC$-derivation of $A$. Then, $D$ is called 
\emph{locally nilpotent}, if
for every $f \in A$ there exists an integer $n = n(f)$ such that $D^n(f) = 0$.
Moreover, $D$ is called \emph{irreducible}, if $D \neq 0$ and the following holds:
if $D = f D'$ for some locally nilpotent derivation $D'$ and some $f \in \ker D$, then
$f \in A^\ast$ where $A^\ast$ denotes the subgroup of units of $A$.

We list some basic facts about locally nilpotent derivations, 
that we will use constantly (see \cite{Fr2006Algebraic-theory-o} for proofs).

\begin{lem} Let $A$ be a $\CC$-algebra and assume it is a UFD, and let $D$
be a locally nilpotent derivation of $A$. Then

\begin{enumerate}[i)]
	\item The units of $A$ lie in $\ker D$.
		In particular, $\CC \subseteq \ker D$.
	\item The kernel $\ker D$ is factorially closed in $A$, i.e. if 
		$f, g \in A$ such that $fg \in \ker D$, then $f, g \in \ker D$.
	\item If $S \subseteq \ker D$ is a multiplicative system, then 
		$D$ extends uniquely to a locally nilpotent derivation of the localization
		$A_S$.
	\item If $D \neq 0$, then there exists 
		$f \in A$ such that $D(f) \in \ker D$ and $D(f) \neq 0$.
	\item If $s \in A$ such that $D(s) = 1$, then 
		$A$ is a polynomial ring in $s$ over
		$\ker D$ and $D = \partial / \partial s$.
	\item For $f \in A$, the derivation $f D$ is locally nilpotent if and only if
		$f \in \ker D$.
	\item If $D$ is irreducible and $E$ is another locally nilpotent derivation
		of $A$ such that $E(\ker D) = 0$, then there exists 
		$f \in \ker D$ such that $E = f D$.
	\item If $D \neq 0$, then there exists a unique irreducible 
		locally nilpotent derivation $D_0$
		(up to multiplication by some element of $A^\ast$) such
		that $\ker D = \ker D_0$.
	\item If $D(f) \in f A$, then $D(f) = 0$.
	\item The exponential $\exp(D) = \sum_{i = 0}^{\infty} D^i / {i!}$
		is a $\CC$-algebra automorphism of $A$ and the map $\exp$ defines
		an injection from the set of locally nilpotent derivations of $A$
		to the set of $\CC$-algebra automorphisms of $A$. 
\end{enumerate}
\end{lem}

\subsection{Uniptent automorphisms of $\CC^3$}
\label{BasicPropertiesOfAUnipotAuto.subsec}


Let $\bu \in \Aut(\CC^3)$ be a unipotent automorphism.
We call $\bu$ \emph{irreducible} if the corresponding
locally nilpotent derivation is irreducible.
If $\bu \neq \bid$, there exists an irreducible
$\bu_0 \in \Aut(\AA^3)$ such that $\bu$ is the modification 
$d \cdot \bu_0$ for some $0 \neq d \in \ker D$. The automorphism
$\bu_0$ is unique up to a modification
by some element in $\CC^\ast$ (cf. subsec.~\ref{BasicPropertiesofLNDs.subsec}). 
We call then 
\[
	\bu = d \cdot \bu_0
\] 
a \emph{standard decomposition}. 

Now, let $\bu$, $\bu'$ be unipotent automorphisms of $\CC^3$ such that
there exists $0 \neq f \in \OO(\CC^3)^\bu$ with $\bu = f \cdot \bu'$.
If $\rho$ and $\rho'$ denote the 
$\CC^+$-actions on $\CC^3$ corresponding to $\bu$ and $\bu'$ respectively, then
the orbits of $\rho$ are contained in the orbits of $\rho'$ and the orbits are
equal on the principal open set $(\CC^3)_f$. Moreover, if $\bu = \Exp(D)$ and
$\bu' = \Exp(D')$, then $\im D = f \im D' \subseteq \OO(\CC^3)$ 
and the plinth divisors $\Gamma$, $\Gamma'$ of $\bu$, $\bu'$ satisfy 
\[
	\Gamma = \Gamma' + \div(f) \, .
\]  

Whenever $S$ is a subset of $\Aut(\CC^n)$ we denote by $S_u$
the subset of unipotent automorphisms in $S$.  

\section{\texorpdfstring{Structure theorems for $\Cent(\bu)$}
	{Structure theorems for Cent(u)}}
\label{Cent.sec}

\subsection{\texorpdfstring{The first unipotent subgroup in $\Cent(\bu)$}
		 {The first unipotent subgroup in Cent(u)}}
\label{The_first_subgroup.subsec}
Let $\bid \neq \bu \in \Aut(\AA^3)$ be unipotent and let 
$\bu = d \cdot \bu_0$ be a standard
decomposition.
There exists an obvious subgroup of unipotent automorphisms in $\Cent(\bu)$: 
The modifications of $\bu_0$, i.e. the subgroup $\OO(\AA^3)^{\bu_0} \cdot \bu_0$.
This subgroup has another characterization:

\begin{prop}
	\label{firstfam.prop}
	Let $\bid \neq \bu \in \Aut(\AA^3)$ be unipotent with standard decomposition 
	$\bu = d \cdot \bu_0$. 
	The subgroup $\OO(\AA^3)^{\bu_0} \cdot \bu_0$ 
	consists of those automorphisms of $\AA^3$ that commute with 
	$\bu$ and that induce the identity on $\AA^3 \aquot \bu$, 
	i.e. the sequence
	\[
		1 \to \OO(\AA^3)^{\bu_0} \cdot \bu_0 
		\hookrightarrow \Cent(\bu)
		\stackrel{p}{\to} \Aut(\AA^3 \aquot \bu, \Gamma) \cap 
		\Aut(\AA^3 \aquot \bu, \Gamma_0)
	\]
	is exact, where $\Gamma$, $\Gamma_0$ denote the plinth divisors of 
	$\bu$, $\bu_0$ respectively.
	Moreover, the homomorphisms in the sequence above are
	homomorphisms of ind-groups.
\end{prop}

This result is an immediate consequence of Remark~\ref{indhomo.rem} 
and Lemma~\ref{identity.lem}.

\begin{rem}
	\label{indhomo.rem}
	Choose generators $v_1, v_2$ of the polynomial ring
	$\OO(\AA^3)^{\bu}$ and choose a $\CC$-linear retraction
	$rÊ\colon \OO(\AA^3) \twoheadrightarrow \OO(\AA^3)^{\bu}$. 
	The map $p \colon \Cent(\bu) \to \Aut(\AA^3 \aquot \bu, \Gamma)$
	is a morphism of ind-varieties due to the following commutative
	diagram
	\[
		\xymatrix{
			\End(\AA^3) \ar[rrr]^-
			{\g \mapsto (\g^\ast(v_1), \g^\ast(v_2))}_-{\textrm{morph.}} 
			&&& \OO(\AA^3)^2 
			\ar@{->>}[r]^-{r \times r}_-{\textrm{lin.}} & (\OO(\AA^3)^{\bu})^2 \\
			\Cent(\bu) \ar@{^(->}[u]_-{\textrm{loc. closed}} 
			\ar[rrrr]^-{p} &&&& \Aut(\AA^3 \aquot \bu, \Gamma)
			\ar@{^(->}[u]^-{\textrm{loc. closed}} \, .
		}
	\]
\end{rem}

\begin{lem}
	\label{identity.lem}
	Let $A$ be a $\CC$-algebra and assume it is a UFD, let
	$B, B'$ be non-zero locally nilpotent derivations of $A$ such that $B'$
	is irreducible and $\ker B = \ker B'$.
	If $\varphi \colon A \to A$ is a $\CC$-algebra automorphism, then we have
	\[
		\varphi |_{\ker B} = \id \ \textrm{and} \ 
		\varphi \circ B = B \circ \varphi \quad \textrm{if and only if} \quad 
		\varphi = \exp(fB') \, , \ f \in \ker B \, .
	\]
\end{lem}

\begin{proof}
	Assume that $\varphi |_{\ker B}$ is the identity and $\varphi$ 
	commutes with $B$. There exists 
	$0 \neq d \in \ker B$, such that $A_d = \ker(B)_d[s]$ is a polynomial ring
	in an element $s \in A_d$ and $B(s) = 1$, if we extend $B$ to $A_d$. 
	Since $\varphi$ commutes with $B$ there exists 
	$g \in \ker(B)_d$ such that the extension $\tilde{\varphi}$ to $A_d$ 
	of $\varphi$ satisfies $\tilde{\varphi}(s) = s + g$. 
	Now, we have $\varphi = \exp(g B) |_A$.
	A density argument shows that
	$\exp(t gB)(A) \subseteq A$ for all $t \in \CC$. Since 
	\[
		gB = \frac{\exp(t gB) - \id}{t} \Big| _{t =0} \, ,
	\]
	we have $gB(A) \subseteq A$. Hence $gB$ is a locally nilpotent derivation
	of $A$ that vanishes on $\ker B = \ker B'$. Thus, $gB = f B'$
	for some $f \in \ker B$. The converse is clear.
\end{proof}

\begin{rem}
	\label{ActionByConj.rem}
	By Proposition~\ref{firstfam.prop}, $\Cent(\bu)$ normalizes 
	$\OO(\CC^3)^{\bu_0} \cdot \bu_0$ and one can easily see, that the action is
	given by
	\[
		\g^{-1} \circ f \cdot \bu_0 \circ \g = \mu(\g) \g^\ast(f) \cdot \bu_0 \, , \quad
		\g \in \Cent(\bu) \, , \ f \in \OO(\CC^3)^{\bu_0} 	
	\]
	where $\mu \colon \Cent(\bu) \to \CC^\ast$ is the homomorphism
	given by $\mu(\g) d = \g^{\ast}(d)$.
\end{rem}

\subsection{\texorpdfstring{Centralizer of a modified 
	    translation in $\Aut(\AA^3)$}
	    {Centralizer of a modified translation in Aut(C3)}}
\label{modtranslation.sec}

%

\begin{prop}
	\label{modtranslation.prop}
	Let $\bid \neq \bu \in \Aut(\AA^3)$ be a modified translation
	with standard decomposition $\bu = d \cdot \bu_0$.
	Denote by $\Gamma$ the plinth divisor of $\bu$. Then
	\[
			1 \to \OO(\AA^3)^{\bu_0} \cdot \bu_0 
			\hookrightarrow \Cent(\bu) \stackrel{p}{\to}
			\Aut(\AA^3 \aquot \bu, \Gamma) \to 1
	\]
	is a split short exact sequence of ind-groups.
	Moreover, there exists a closed subgroup of $\Cent(\bu)$
	that is mapped via $p$ isomorphically onto 
	$\Aut(\AA^3 \aquot \bu, \Gamma)$.
\end{prop}

\begin{rem}
	\label{AlgebraicElements.rem}
	Under the assumptions of Proposition~\ref{modtranslation.prop},
	$\Cent(\bu)$ consists only of algebraic elements, provided that 
	$\Gamma$ is non-empty. Indeed, let $H \subseteq \Cent(\bu)$
	be a closed subgroup, such that $p$ induces an isomorphism
	$H \simeq \Aut(\CC^3 \aquot \bu, \Gamma)$. 
	Let $(f \cdot \bu_0, \h) \in \OO(\CC^3)^{\bu_0} 
	\cdot \bu_0 \rtimes H \simeq \Cent(\bu)$. By Proposition~\ref{curves.prop} i),
	$R = \overline{\langle \h \rangle}$ is an algebraic subgroup of $H$.
	Hence
	\[
		W = \Span \{ \, \br^\ast(f) \ | \ \br \in R \, \}
	\]
	is a finite dimensional subspace of $\OO(\CC^3)^{\bu_0}$ and
	$W \cdot \bu_0 \rtimes R \subseteq \OO(\CC^3)^{\bu_0} \cdot \bu_0 \rtimes H$
	is an algebraic subgroup  that contains $(f \cdot \bu_0, \h)$
	(see Remark~\ref{ActionByConj.rem}).
\end{rem}

\begin{proof}[Proof of Proposition~\ref{modtranslation.prop}]
	By Proposition~\ref{firstfam.prop}, the sequence is left
	exact.
	By assumption, there exist coordinates $(x, y, z)$ on $\CC^3$
	such that $d \in \CC[y, z]$ and $\bu = (x+d, y, z)$. Moreover, 
	we can identify the algebraic
	quotient $\pi \colon \AA^3 \to \CC^3 \aquot \bu = \AA^2$ 
	with the map $(x, y, z) \mapsto (y, z)$
	and $\Gamma = \div(d)$.
	Let $\f \in \Aut(\AA^2, \Gamma)$. Then $\f^\ast(d) = \lambda(\f) d$
	for some $\lambda(\f) \in \CC^\ast$.
	One can see that 
	$\Aut(\AA^2, \Gamma) \to \CC^\ast$, $\f \mapsto \lambda(\f)$
	is a homomorphism of ind-groups. Thus
	\[
		H = \{ \, \sigma \in \Aut(\CC^3) \ | \ \sigma(x, y, z) = 
		(\lambda x, \f(y, z)) \ \textrm{with} \ \f \in 
		\Aut(\AA^2, \Gamma) \ \textrm{and} \ \lambda(\f) = \lambda 
		\, \}
	\]
	is a closed subgroup of $\Cent(\bu)$ 
	(note that the subgroup $\Aut(\AA^2, \Gamma) \subseteq \Aut(\AA^2)$ 
	is closed by Proposition~\ref{curves.prop}) and
	$p|_H \colon H \to \Aut(\AA^2, \Gamma)$ is an isomorphism of ind-groups.
\end{proof}

\begin{prop}
	Let $\bu \in \Aut(\CC^3)$ be unipotent. Then 
	$\Cent(\bu)$ contains $(\CC^\ast)^2$ as an algebraic subgroup
	if and only if $\bu$ is a modified translation and the plinth divisor
	$\Gamma$ is given by $v^i w^j$ for some coordinates $(v, w)$
	of $\CC^3 \aquot \bu \simeq \CC^2$.
\end{prop}

\begin{proof}
	Assume that $\Cent(\bu)$ contains an algebraic subgroup 
	$T \simeq (\CC^\ast)^2$.
	By Proposition~\ref{modtranslation.prop} it follows
	that $T$ acts faithfully on $\CC^3 \aquot \bu$ and preserves the plinth
	divisor $\Gamma$. Hence, it follows from
	\cite{Bi1966Remarks-on-the-act} that
	there exist coordinates $(v, w)$ on $\CC^2 \simeq \CC^{3} \aquot \bu$
	such that $\Gamma$ is given by $v^i w^j$ for some integers $i, j$.
	By \cite{Bi1967Remarks-on-the-act} there exist coordinates
	$(x_1, x_2, x_3)$ of $\CC^3$ such that the action of $T$ is diagonal
	with respect to these coordinates. Hence there exist characters
	$\lambda_1$, $\lambda_2$, $\lambda_3$ of $T$ such that
	$\bt(x_1, x_2, x_3) = 
	(\lambda_1(\bt)x_1, \lambda_2(\bt)x_2, \lambda_3(\bt)x_3)$ for all 
	$\bt \in T$. Let $\bu = \Exp(D)$. By assumption we have for all
	$\bt \in T$ and $i = 1, 2, 3$
	\begin{equation}
		\label{Comm.eq}
		D(x_i) \circ \bt = \lambda_i(\bt) D(x_i).
	\end{equation}
	As the action of $T$ on $\CC^3$
	is faithful, the subgroup spanned by $\lambda_1, \lambda_2, \lambda_3$
	inside the characters of $T$ has rank $2$.
	Assume first that the $\lambda_i$ are pairwise different. 
	Then there exist at least two different indices $k_1, k_2 \in \{1, 2, 3 \}$ 
	such that $\lambda_{k_i}$ lies not in the monoid
	spanned by $\{ \, \lambda_l \ | \ l \neq k_i \, \}$. By symmetry we can 
	assume $k_1 = 1$, $k_2 = 2$.
	This implies that $D(x_i) \in x_i \CC[x_1, x_2, x_3]$ for $i = 1, 2$.
	Since $D$ is locally nilpotent we have $D(x_1)=D(x_2) = 0$.
	Hence, $\bu$ is a modified translation.	
	Assume now that $\lambda_1 = \lambda_2 \neq \lambda_3$
	(the other cases follow by symmetry). Thus, 
	$\lambda_3$ does not lie in the monoid spanned by $\lambda_1$
	and $\lambda_2$. Hence we get $D(x_3) \in x_3 \CC[x_1, x_2, x_3]$ and
	$D(x_1), D(x_2) \in \CC x_1 \oplus \CC x_2$. Since $D$ is locally nilpotent
	it follows that $D(x_3) = 0$ and the linear endomorphism
	$D |_{\CC x_1 \oplus \CC x_2}$ is nilpotent. This implies that $\bu$ is 
	a modified translation.
	
	The converse of the statement is clear.
\end{proof}


\subsection{\texorpdfstring{The second unipotent subgroup in $\Cent(\bu)$}
	          {The second unipotent subgroup in Cent(u)}}
\label{The_second_subgroup.subsec}
Let $\bid \neq \bu \in \Aut(\AA^3)$ be unipotent with standard decomposition 
$\bu = d \cdot \bu_0$. 
Throughout this subsection we assume that the plinth divisor $\Gamma = \div(a)$ 
of $\bu$ is a fence. There exists another subgroup of unipotent automorphisms 
inside $\Cent(\bu)$ in addition to $\OO(\AA^3)^{\bu_0} \cdot \bu_0$, that we describe
in this subsection.
\begin{lem}
	\label{fact.lem}
	Let $\bid \neq \bu \in \Aut(\AA^3)$ be unipotent. If the 
	plinth divisor $\Gamma = \div(a)$ is a fence,
	then there exists a variable $z$ of $\OO(\AA^3)^{\bu} = \OO(\CC^2)$ 
	such that $a \in \CC[z]$ and any such $z$ is a variable of 
	$\OO(\AA^3)$.
\end{lem}

\begin{proof}
	By Proposition~\ref{genAMS.prop} there exists a coordinate system 
	$(z, w)$ of $\AA^2$ such that the embedding 
	$\div(a) = \Gamma \subseteq \AA^2$ is given by the standard embedding
	$F \times \AAone \subseteq \AA^2$ for some 0-dimensional
	closed subscheme $F$ of $\AAone$. Thus $a \in \CC[z]$.
	Since the quotient map $\pi \colon \CC^3 \to \CC^3 \aquot \bu$ 
	is a trivial $\AAone$-bundle over $\AA^2 \setminus \Gamma$, 
	it follows that only finitely many fibers of 
	$z \colon \AA^3 \to \AAone$ are non-isomorphic
	to $\AA^2$. Thus $z$ is a variable of $\OO(\CC^3)$, according to 
	\name{Kaliman}'s Theorem \cite{Ka2002Polynomials-with-g}.
\end{proof}

\begin{rem}
	\label{rank2.rem}
	If $\bu = \Exp(D)$ is irreducible, then
	$\Gamma$ is a fence
	if and only if $\rank D \leq 2$ (i.e. there exists a variable $z$ of $\OO(\CC^3)$
	that lies in $\ker D$). This follows from the lemma above
	and from 
	\cite[Theorem~2.4, Proposition~2.3]{DaFr1998Locally-nilpotent-}.
\end{rem}

\begin{defn}
	Let $A$ be a UFD and let $P \in A[x, y]$. We denote
	\[
		\Delta_P = -P_y \frac{\partial}{\partial x}
				  + P_x \frac{\partial}{\partial y}
	\]
	where $P_x$ and $P_y$ denote the partial derivatives of $P$ with respect
	to $x$ and $y$ respectively. Obviously, $\Delta_P$ is an 
	$A$-derivation of $A[x, y]$ and $\Delta_P(P) = 0$.
\end{defn}

Let $D$, $D_0$ be locally nilpotent derivations of $\OO(\AA^3)$ 
such that $\bu = \Exp(D)$ and $\bu_0 = \Exp(D_0)$. 
Let $z \in \ker D$ be a variable such that $a \in \CC[z]$ and let
$(x, y, z)$ be a coordinate system of $\OO(\CC^3)$ (see Lemma~\ref{fact.lem}). 
Let $A = \CC[z]$.
It follows now from \cite[Theorem~2.4]{DaFr1998Locally-nilpotent-} that there exists 
$P \in A[x, y]$ such that
\[
	D_0 = \Delta_P \quad \textrm{and} \quad \ker D = \ker D_0 = \CC[z, P] \, .
\]
Obviously, $d$ divides $a$ in $\CC[z]$. 
Let $a = d a_0$. An easy calculation shows that $\div(a_0)$ is the plinth 
divisor of $\bu_0$ and that for all $Q \in A[x, y]$ we have
\begin{equation}
	\label{D_0.eq}
	D(Q) = a \quad \textrm{if and only if} \quad D_0(Q) = a_0 \, .
\end{equation}
By assumption $\Gamma = \div(a)$ is a fence and thus $a, a_0 \neq 0$.

\begin{lem}
	\label{AdmissibleCompl.lem}
	Let $A = \CC[z]$. If $Q \in A[x, y]$ such that $D(Q) = a$, then $E = \Delta_Q$
	is an irreducible locally nilpotent derivation.
	Moreover, $E$ commutes with $D$.
\end{lem}

\begin{proof}
	Let $K$ be the quotient field of $A$.
	The extension of $D_0$ to $K[x, y]$ satisfies $D_0(Q/a_0)= 1$.
	Thus $K[x, y] = K[P, Q]$. $E$ is non-zero, since $E(P) = -a_0 \neq 0$.
	If we extend $E$ to a 
	derivation of $K[x, y]$ one easily sees that 
	$E$ is locally nilpotent. Thus $E$ is a non-zero locally nilpotent derivation of 
	$A[x, y]$.
	
	By \cite[Theorem~2.4, Proposition~2.3]{DaFr1998Locally-nilpotent-}) 
	there exists $S \in A[x, y]$ and $0 \neq h \in A[P]$ such that 
	$E = h \Delta_{S}$ and $\Delta_{S}$ is irreducible. Thus
	$-a_0 = E(P) = h \Delta_{S}(P) = - h \Delta_P(S)$. Hence 
	$\Delta_P(S)$ lies in the plinth ideal of $\Delta_P$ and thus 
	$\Delta_P(S)$ is a multiple of $a_0$. This implies that
	$h \in \CC^\ast$ and proves that $E$ is irreducible.
	
	If we extend $\Delta_P$ and $\Delta_Q$ to $K[x, y] = K[P, Q]$, 
	we get $\Delta_P = a_0 (\partial / \partial Q)$ and
	$\Delta_Q = -a_0 (\partial / \partial P)$.
	Thus $E$ commutes with $D_0$. Since  $d \in \CC[z]$, 
	$E$ and $D = d D_0$ commute.
\end{proof}

\begin{defn}
For any $Q \in \OO(\CC^3)$ with $D(Q) = a$ we call
\[
	\e = \Exp(E) = \Exp(\Delta_Q) 
\]
an \emph{admissible complement to $\bu$}.
\end{defn}

By \eqref{D_0.eq}, we get that $\e$ is an admissible complement to $\bu$
if and only if $\e$ is an admissible complement to $\bu_0$.
It follows from Lemma~\ref{AdmissibleCompl.lem} that 
$\OO(\AA^3)^{\langle \e, \bu \rangle} \cdot \e$ is a subgroup of 
unipotent automorphisms inside $\Cent(\bu)$. 
    
\begin{rem}Ê$\textrm{}$
	\label{geom.rem}
	We have $\AA^2 \setminus \Gamma = U \times \AAone$
	for some non-empty open subset $U \subseteq \AAone$.
	The restriction of $\bu$ and of $\e$ to the open subset 
	\[
		\pi^{-1}(\AA^2 \setminus \Gamma) = 
		(U \times \AAone) \times \AAone = \Spec( \CC[z]_a[P, Q])
	\]
	are given by $(u, v, w) \mapsto (u, v, w + 1)$ and 
	$(u, v, w) \mapsto (u, v + 1, w)$ respectively, where $(u, v, w)$
	is the coordinate system $(z, -P/a_0, Q/a)$.
\end{rem}

\subsection{\texorpdfstring{The property (Sat)}
		 {The property (Sat)}}
We introduce in this subsection a property for a subset
$S \subseteq \Aut(\AA^n)$ and 
we will show that $\Cent(\bv)$ satisfies this property for any
unipotent automorphism $\bv \in \Aut(\AA^n)$.
This property will then play a key role when we describe the 
set of unipotent elements inside the centralizer. One can think of this property
as a saturation feature on the unipotent elements in $S$.

\begin{defn}
	Let $S \subseteq \Aut(\AA^n)$ be a subset.
	We say that $S$ has the \emph{property (Sat)}
	if for all unipotent $\w \in \Aut(\AA^n)$ 
	and for all $0 \neq f \in \OO(\AA^n)^{\w}$ we have
	\[
		\tag{Sat}
		f \cdot \w \in S \quad \Longrightarrow \quad \w \in S \, .
	\]
\end{defn}

\begin{prop}
	\label{propertyR.prop}
	If $\bv \in \Aut(\AA^n)$ is unipotent, then 
	the subgroup $\Cent(\bv) \subseteq \Aut(\AA^n)$ satisfies the property (Sat).
\end{prop}

\begin{proof}
	Let $\bv = \Exp(B)$ and let $\w = \Exp(F)$.
	Assume that 
	$f \cdot \w$ commutes with $\bv$ for some $\w$-invariant 
	$0 \neq f \in \OO(\AA^n)$. If $\bv = \bid$ or $\w = \bid$,
	then (Sat) is obviously satisfied. Thus we assume $\bv \neq \bid \neq \w$.
	For the Lie-bracket we have
	\begin{equation}
		\label{Lie-bracket.eq}
		0 = [fF, B] = f [F, B] - B(f) F \, .
	\end{equation}
	Thus, it is enough to prove that $B(f) = 0$.
	
	First, assume that $F$ is irreducible.
	By \eqref{Lie-bracket.eq}, it follows that $f$ divides $B(f) F(g)$
	for all $g \in \OO(\AA^n)$. As $F$ is irreducible, it follows that
	$f$ divides $B(f)$.
	Since $B$ is locally nilpotent, it follows that $B(f) = 0$.

	Now, let $F = f' F'$ for some irreducible $F'$. Thus, $f f' F'$ commutes with
	$B$ and by the argument above, $B(ff') = 0$. 
	Since $\ker B$ is factorially closed in $\OO(\AA^n)$, we have $B(f) =0$.
\end{proof}

\subsection{\texorpdfstring{The subgroup $N \subseteq \Cent(\bu)$}
		 {The subgroup N in Cent(u)}}
\label{TheSubgroupN.subsec}
Let $\bid \neq \bu \in \Aut(\AA^3)$ be unipotent.
We define in this subsection a subgroup $N$ of $\Cent(\bu)$ 
and we gather some facts about this group.
In the next subsection, we will prove that $N$ is exactly the set of unipotent
automorphisms $\Cent(\bu)$ if $\bu$ is not a translation.

\begin{defn} 
Let $\bid \neq \bu \in \Aut(\AA^3)$ be unipotent
with standard decomposition $\bu = d \cdot \bu_0$ and let 
$\Gamma$ be the plinth divisor of $\bu$. Let
\[
	 N = N(\bu) = \left\{ 
			\begin{array}{rl}
				\OO(\AA^3)^{\langle \e, \bu_0 \rangle} \cdot \e 
				\circ \OO(\AA^3)^{\bu_0} \cdot \bu_0 & 
				\textrm{if $\Gamma$ is a fence} \\
				\OO(\AA^3)^{\bu_0} \cdot \bu_0 & \textrm{otherwise.}
			\end{array}
		     \right.
\]
where $\e$ is an admissible complement to $\bu$
(cf. subsec.~\ref{The_second_subgroup.subsec}). Moreover, let
\[
	M = M(\bu) = \left\{ 
			\begin{array}{rl}
				(\ker E \cap \ker D_0) E + \ker(D_0) D_0 & 
				\textrm{if $\Gamma$ is a fence,} \\
				\ker(D_0) D_0 & \textrm{otherwise.}
			\end{array}
		     \right.
\]
where $\bu_0 = \Exp(D_0)$ and 
$\e = \Exp(E)$ (cf. Lemma~\ref{AdmissibleCompl.lem}).
\end{defn}

\begin{prop}
	\label{properties.prop}
	Let $\bid \neq \bu \in \Aut(\AA^3)$ be unipotent
	and assume it is not a translation. Then:
	\begin{enumerate}[i)]
		\item $N$ consists of 
		unipotent automorphisms and we have $N = \Exp(M)$.

		\item $N$
		normalizes $\OO(\AA^3)^{\bu_0} \cdot \bu_0$ 
		and we have for all $\g \in N$ and for all
		$f \in \OO(\AA^3)^{\bu_0}$
		\[
			\g^{-1} \circ f \cdot \bu_0 \circ \g = \g^\ast(f) \cdot \bu_0 \, .
		\]
		
		\item $N$ is a closed normal subgroup of $\Cent(\bu)$ that fits into
		the following split short exact sequence of ind-groups
		\[
				1 \to \OO(\AA^3)^{\bu_0} \cdot \bu_0 
				\hookrightarrow N 
				\stackrel{p|_N}{\longrightarrow}
				\Aut(\AA^2, \Gamma) \cap 
				\Iner(\AA^2, \Gamma_0)_u  
				\to 1 \, ,
		\]
		where $\Gamma_0$ is the plinth divisor of $\bu_0$.
		If $\Gamma$ is a fence, then the restriction of
		$p$ to ${\OO(\AA^3)^{\langle \bu_0, \e \rangle} \cdot \e}$ 
		is an isomorphism of ind-groups.
		In particular, $N$ is independent of the choice of $\e$.
		
		\item $N \subseteq \Aut(\AA^3)$ 
		satisfies the property (Sat).
	\end{enumerate}
\end{prop}

\begin{proof}
	Assume first that
	$\Gamma$ is not a fence. Then i) ii) and iv) are clear, iii) follows from
	Proposition~\ref{curves.prop}. Thus we can assume that
	$\Gamma$ is a fence.
	\begin{enumerate}[i)]
		\item Let $hE + fD_0 \in M$. By induction on $l \geq 1$ one sees that 
		$(hE + fD_0)^l$ is a sum of terms of the form $g E^i (D_0)^j$
		where $g \in \ker D_0$. From this fact, one can deduce that $hE + fD_0$
		is locally nilpotent and hence $M$ consists only of locally nilpotent 
		derivations. 
		
		
		For all $f \in \ker D_0$ and $h \in \ker D_0 \cap \ker E$ and $q \geq 0$ 
		we have 
		\[
			fD_0 \ad(hE)^{q} 
			= (-1)^q h^qE^q(f) D_0
		\]
		where $A\ad(B) = [A, B]$.
		With the aid of this formula, an application of the 
		Baker-Campbell-Hausdorff formula yields 
		$\Exp(hE) \circ \Exp(fD_0) \in M$
		(see \cite[Proposition 1, \S 5, chp. V]{Ja1962Lie-algebras}).
		Hence $N \subseteq \Exp(M)$ which shows in particular, that 
		$N$ consists of unipotent automorphisms. 
		Moreover, 
		$\exp hE$ and $\exp(fD_0 + hE)$ coincide
		on $\ker D_0$. 
		Lemma~\ref{identity.lem} implies
		$(\Exp hE)^{-1} \circ \Exp(fD_0 + hE) = \Exp(g D_0)$ for some
		$g \in \ker D_0$ and thus $\Exp(M) \subseteq N$.		
				
		\item This follows from Remark~\ref{ActionByConj.rem}.
		
		\item 
		One can check that 
		$N = p^{-1}(\Aut(\AA^2, \Gamma) \cap \Iner(\AA^2, \Gamma_0)_u)$ by
		using Proposition~\ref{firstfam.prop}.
		Since $\Aut(\AA^2, \Gamma) \cap \Iner(\AA^2, \Gamma_0)_u$ is a 
		closed normal subgroup of
		$\Aut(\AA^2, \Gamma) \cap \Aut(\AA^2, \Gamma_0)$ 
		it follows that $N$ is a closed normal 
		subgroup of $\Cent(\bu)$.
		
		It is enough to show that the homomorphism
		$\OO(\AA^3)^{\langle \e, \bu_0 \rangle} \cdot \e \to 
		\Aut(\AA^2, \Gamma) \cap \Iner(\AA^2, \Gamma_0)_u$ (induced by $p$)
		is an isomorphism of ind-groups. Injectivity follows from the fact that
		$\OO(\AA^3)^{\langle \e, \bu_0 \rangle} \cdot \e \cap 
		\OO(\AA^3)^{\bu_0} \cdot \bu_0 = \{ \bid \}$ 
		and surjectivity follows from a straightforward calculation, by using that
		$\Gamma$ is non-empty. The inverse map is clearly a morphism. 
		
		\item Let $0 \neq hE + fD_0 \in M$. It is enough to prove that
		\begin{equation}
			\label{triangle}
			\gcd(h, f) = 1 \quad \Longrightarrow \quad
			\textrm{$hE + fD_0$ is irreducible}
		\end{equation}
		where the greatest common divisor
		is taken in the polynomial ring $\ker D_0 = \CC[z, P]$
		(we use the notation of subsec.~\ref{The_second_subgroup.subsec}).
		Indeed, let $g B = hE + fD_0 \in M$ for some
		locally nilpotent derivation
		$B \neq 0$ and some $0 \neq g \in \ker B$ and let $h = \gcd(h, f) h_0$,
		$f = \gcd(h, f) f_0$. Thus $B$ vanishes on $\ker(h_0 E + f_0 D_0)$ and
		since $h_0 E + f_0 D_0$ is irreducible, there exists 
		$b \in \ker(h_0 E + f_0 D_0)$
		such that $B = b (h_0 E + f_0 D_0)$. 
		This implies $g b = \gcd(h, f) \in \CC[z]$
		and therefore $b \in \CC[z]$. This shows that $B \in M$.
		
		Let us prove \eqref{triangle}.
		Since $E$ and $D_0$ are irreducible 
		(see Lemma~\ref{AdmissibleCompl.lem}) we can
		assume that $h$ and $f$ both are non-zero. A calculation shows
		\[
			h E + f D_0 = \Delta_{F} \, , \quad 
			F = hQ + fP -Ê\int \left( \frac{\partial f}{\partial P} P \right) \ud P
		\]
		where the integration is taken inside the polynomial ring 
		$\ker D_0 =\CC[z, P]$ and $\Delta_F$
		is taken with respect to $A[x, y]$ where $A = \CC[z]$. 
		Let $f = \sum_{i=0}^n f_i(z) P^i$. Thus we have
		\[
			f P - \int \left( \frac{\partial f}{\partial P} P \right) \ud P =
			\sum_{i=0}^n f_i(z) \left( 1 - \frac{i}{i+1} \right) P^{i+1} \, .
		\]
		Denote this last polynomial by $G \in \CC[z, P]$.
		
		Now, assume towards a contradiction that $hE + fD_0$ is not irreducible.
		Hence, we have $hE + fD_0 = bB$ for some locally 
		nilpotent derivation $B$ and some 
		non-constant $b \in \ker B$. By plugging in 
		$P$ and $Q$ in $hE + fD_0 = bB$ 
		and using the fact that $\gcd(h, f) = 1$ 
		we see that $b$ divides $a_0$
		(recall that $D_0(Q) = a_0$ and $E(P) = -a_0$).
		Hence there exists a root $z_0$ of $a_0$ such that the induced
		derivation of $\Delta_F = hE + fD_0$ on 
		$\CC[x, y, z] / (z-z_0) \simeq \CC[x, y]$
		vanishes. Thus, there exists a constant $c \in \CC$ such that
		\begin{equation}
			\label{square}
			h(z_0) Q(x, y, z_0) + 
			\sum_{i = 0}^n 
			f_i(z_0) \left( 1 - \frac{i}{i+1} \right) P^{i+1}(x, y, z_0) =  c \, . 
		\end{equation}
		The polynomial $P(x, y, z_0) \in \CC[x, y]$ is non-constant,
		since otherwise $\bu = \Exp(\Delta_P)$ would have a 
		two-dimensional fixed point set, 
		contradicting the irreducibility
		(cf. \cite[2.10]{Da2007On-polynomials-in-}).
		If $h(z_0) = 0$, then we have $f(z_0, P) = 0$ by \eqref{square}.
		Hence $\gcd(h, f) \neq 1$, a contradiction. Thus we can assume
		$h(z_0) \neq 0$. It follows that
		$Q + h(z_0)^{-1} (G(z, P) - c)$ is divisible by $z - z_0$ 
		inside $\OO(\CC^3)$. Thus,
		\[
			D_0 \left( \frac{Q + h(z_0)^{-1}(G(z, P)-c)}{z-z_0} \right) = 
			\frac{a_0}{z-z_0} \, .
		\]
		But this contradicts the fact, that $a_0$ is a generator
		of the plinth ideal of $D_0$.
	\end{enumerate}
\end{proof}

\subsection{\texorpdfstring{The group $\Cent(\bu)$ as a semi-direct product}
		{The group Cent(u) as a semi-direct product}}
In this subsection, we prove our first main result: There exists an algebraic 
subgroup $R \subseteq \Cent(\bu)$ such that $\Cent(\bu)$
is the semi-direct product of $N$ with $R$, if $\bu$ is not a translation.

\begin{thm}
	\label{main.thm1}
	Let $\bu \in \Aut(\AA^3)$ be unipotent and assume that
	$\bu$ is not a translation. Then the subgroup 
	$N \subseteq \Cent(\bu)$ is closed and normal, and
	there exists an algebraic subgroup $R \subseteq \Cent(\bu)$
	such that $\Cent(\bu) \simeq N \rtimes R$
	as ind-groups. Moreover, all elements of $\Cent(\bu)$ are algebraic.
\end{thm}

We prove the result for modified translations and reduce the general case to it.

\begin{proof}[Proof for a modified translation]
	Let $\bu = d \cdot \bu_0$ be a standard decomposition. There exists 
	a coordinate system $(x, y, z)$ such that 
	$\bu_0(x, y, z) = (x + 1, y, z)$ and
	$d \in \CC[y, z] \setminus \CC$.
		
	If $\Gamma$ is not a fence, then 
	it follows from Proposition~\ref{curves.prop} 
	that $\Aut(\AA^2, \Gamma)$ is an algebraic group. By 
	Proposition~\ref{modtranslation.prop} there exists a 
	closed subgroup $R$ of $\Cent(\bu)$ that is mapped via 
	$p \colon \Cent(\bu) \twoheadrightarrow \Aut(\AA^2, \Gamma)$
	isomorphically onto $\Aut(\AA^2, \Gamma)$ and 
	$\Cent(\bu) \simeq N \rtimes R$.
	
	Now, assume that $\Gamma = \div(a)$ is a non-empty fence.
	By Proposition~\ref{genAMS.prop} there exist coordinates $(y, z)$ of 
	$\AA^2 = \AA^3 \aquot \bu$ such that $a \in \CC[z]$.
	Thus we have a split short exact sequence of ind-groups
	\[
		1 \to \Aut(\AA^2, \Gamma)_u 
		   \hookrightarrow \Aut(\AA^2, \Gamma)
		   \stackrel{q}{\to} 
		   \CC^{\ast} \times \Aut(\AAone, V(a)) 
		   \to 1
	\]
	where $q$ sends an automorphism 
	$(y, z) \mapsto (\lambda y + h, \alpha z + \beta)$ to 
	$(\lambda, z \mapsto \alpha z + \beta)$.
	Let $R$ be the algebraic group $\CC^{\ast} \times \Aut(\AAone, V(a))$.
	Since $N$ is generated by $\OO(\AA^3)^{\bu_0} \cdot \bu_0$ 
	and $\OO(\AA^3)^{\langle \e, \bu_0 \rangle} 
	\cdot \e \simeq \Aut(\AA^2, \Gamma)_u$ 
	(see Proposition~\ref{properties.prop}),
	we have the desired split short exact sequence of ind-groups
	\[
		1 \to N \hookrightarrow 
		\Cent(\bu) \stackrel{q \circ p}{\longrightarrow}
		R \to 1 \, .
	\]	
	By Remark~\ref{AlgebraicElements.rem}, every element of 
	$\Cent(\bu)$ is algebraic.
\end{proof}

\begin{proof}[Proof in the general case]
	Let $\OO(\AA^3)^{\bu} = \CC[\tilde{y}, \tilde{z}]$ and let 
	$\tilde{x} \in \OO(\CC^3)$ 
	such that $\bu^\ast(\tilde{x}) = \tilde{x} + a$, where
	$\Gamma = \div(a)$.
	Let $\tilde{\bu}  \in \Aut(\AA^3)$ be given by 
	$\tilde{\bu}(\tilde{x}, \tilde{y}, \tilde{z}) = (\tilde{x} + a, \tilde{y}, \tilde{z})$ 
	where we interpret
	$a$ as a polynomial in $\tilde{y}$ and $\tilde{z}$. 
	The morphism $\AA^3 \to \AA^3$
	induced by the inclusion 
	$\CC[\tilde{x}, \tilde{y}, \tilde{z}] \subseteq \OO(\AA^3)$ 
	is birational and thus we get an injective group homomorphism
	\[
		\eta \colon \Cent(\bu) \longrightarrow \Cent(\tilde{\bu}) \, .
	\]
	In fact, $\eta$ is a homomorphism of ind-groups, due to the following
	commutative diagram, where 
	$r \colon \OO(\CC^3) \twoheadrightarrow \CC[\tilde{x}, \tilde{y}, \tilde{z}]$ 
	is a $\CC$-linear retraction
	\[
		\xymatrix{
			\End(\CC^3) \ar[rrrr]^-
			{\g \mapsto (\g^\ast(\tilde{x}), \g^\ast(\tilde{y}), 
			\g^\ast(\tilde{z}))}_-{\textrm{morph.}} 
			&&&& \OO(\CC^3)^3 
			\ar@{->>}[rr]^-{r \times r \times r}_-{\textrm{lin.}} 
			&& \CC[\tilde{x}, \tilde{y}, \tilde{z}]^3 \\
			\Cent(\bu) \ar@{^(->}[u]_-{\textrm{loc. closed}} 
			\ar[rrrrrr]^-{\eta} &&&&&& \Cent(\tilde{\bu})
			\ar@{^(->}[u]^-{\textrm{loc. closed}}
		}
	\]	
	According to the first case, $\Cent(\tilde{\bu})$ is the semi-direct 
	product of $N(\tilde{\bu})$ with some algebraic subgroup 
	$\tilde{R} \subseteq \Cent(\tilde{\bu})$. Let $H \subseteq \tilde{R}$ be
	an algebraic subgroup. We claim that 
	$\eta^{-1}(H) \subseteq \Cent(\bu)$ is an algebraic subgroup.
	Since $\eta \colon \Cent(\bu) \to \Cent(\tilde{\bu})$ is a 
	homomorphism of ind-groups, it
	follows that $\eta^{-1}(H)$ is a closed subgroup. As $H$ is algebraic
	and thus acts locally finite on $\AA^3$, it follows 
	that $\eta^{-1}(H)$ acts also locally finite on $\AA^3$ by
	\cite[Lemma~3.6]{KrSt2013On-Automorphisms-o}. This implies the claim.

	According to the claim all elements of $\Cent(\bu)$ are algebraic
	and $R = \eta^{-1}(\tilde{R})$ is algebraic as well.
	Since $\eta$ is an injective homomorphism of ind-groups we have the 
	following commutative diagram
	\[
		\xymatrix{
			1 \ar[r] & N(\tilde{\bu}) \ar@{^(->}[r] 
			& \Cent(\tilde{\bu}) 
			\ar@{->>}[r] & \tilde{R} \ar[r] & 1 \\
			1 \ar[r] & N(\bu) \ar[u]^-{\textrm{iso. of groups}} 
			\ar@{^(->}[r] & 
			\Cent(\bu) \ar[u]_{\eta} \ar@{->>}[r] & R
			\ar@{^(->}[u]_-{\textrm{cl. embedd.}} \ar[r] & 1
		}
	\] 
	As the first column is a split short exact sequence of 
	ind-groups, the second coloumn is also a split short exact sequence of 
	ind-groups. This proves the theorem.
\end{proof}

\subsection{\texorpdfstring{The unipotent elements of $\Cent(\bu)$}
		 {The unipotent elements of Cent(u)}}
The goal of this subsection is to prove our second main result: 
The unipotent elements of $\Cent(\bu)$ are exactly $N$ provided $\bu$
is not a translation (see subsec.~\ref{TheSubgroupN.subsec} 
for the definition of $N$).
As we know from Proposition~\ref{properties.prop} the set $N$ 
satisfies the property (Sat). This
will be a key ingredient in the proof.

\begin{thm}
   	\label{main.thm2}
	Let $\bid \neq \bu \in \Aut(\AA^3)$ be unipotent and assume
	it is not a translation. Then the
	set of unipotent elements of $\Cent(\bu)$ is equal to $N$.
\end{thm}

\begin{proof}
	Let $\bu = d \cdot \bu_0$ be a standard decomposition.
	Let $\g \in \Cent(\bu)$ be a unipotent automorphism with $\g \neq \bid$. 
	If $\Gamma$ is not a fence, then
	$\Aut(\AA^2, \Gamma)$ contains no unipotent
	automorphism $\neq \bid$ (see Proposition~\ref{curves.prop}). 
	By Proposition~\ref{firstfam.prop} it follows that
	$\g \in \OO(\AA^3)^{\bu_0} \cdot \bu_0 = N$.
	
	Hence we can assume that $\Gamma = \div(a)$ is a fence.
	Let $z \in \OO(\AA^3)^{\bu}$ be a variable of $\OO(\CC^3)$
	such that $a \in \CC[z]$ (see Lemma~\ref{fact.lem}).
	If $\OO(\AA^3)^{\g} = \OO(\AA^3)^{\bu}$, then
	$\g$ is a modification of $\bu_0$ and therefore $\g \in N$.  	
	Now, assume $\OO(\AA^3)^{\g} \neq \OO(\AA^3)^{\bu}$.
	Thus $\g$ is not a modification of $\bu_0$ and hence
	$\bid \neq p(\g) \in \Aut(\AA^2, \Gamma)$. Since $p(\g)$ is unipotent and 
	$0 \neq a \in \CC[z]$, it follows that $z \in \OO(\AA^3)^{\g}$ and thus	
	$\OO(\AA^3)^{\langle \g, \bu \rangle}$ is an $\infty$-dimensional
	$\CC$-vector space. 
	By Theorem~\ref{main.thm1} there exists
	an algebraic subgroup $R \subseteq \Cent(\bu)$ 
	and a split short exact sequence of ind-groups
	\[
		1 \to N \hookrightarrow \Cent(\bu) \stackrel{r}{\twoheadrightarrow} 
		R \to 1 \, .
	\]
	If $\g \notin N$, then $\OO(\AA^3)^{\langle \g, \bu \rangle} \cdot \g \cap N = 1$,
	since $N$ satisfies the property (Sat). 
	We get an injection
	$\OO(\AA^3)^{\langle \g, \bu \rangle} \to \Cent(\bu) 
	\twoheadrightarrow R$, $h \mapsto r(h \cdot \g)$.
	Choose any filtration by finite dimensional $\CC$-subspaces
	to turn $\OO(\AA^3)^{\langle \g, \bu \rangle}$
	into an ind-group. It follows that 
	$\OO(\AA^3)^{\langle \g, \bu \rangle} \to R$ is an injective 
	homomorphism of ind-groups.
	But this implies that $R$ has algebraic subgroups
	of arbitrary high dimension, which is absurd. 
	This finishes the proof of the theorem.
\end{proof}

If we endow $\Cent(\bu) / N$ with the algebraic group structure induced by
the semi-direct product decomposition coming from Theorem~\ref{main.thm1}, 
then we get immediately the following
corollary from Theorem~\ref{main.thm2}.

\begin{cor}
	\label{mainthm2.cor}
	If $\bu \in \Aut(\AA^3)$ is unipotent and not a translation,
	then the algebraic group $\Cent(\bu) / N$ consists only of semi-simple
	elements. In particular, the connected component of the neutral element
	in $\Cent(\bu) / N$ is a torus.
\end{cor}

\section{\texorpdfstring{Applications}{Applications}}

\begin{prop}
	\label{applicationfence.prop}
	Let $\bu \in \Aut(\AA^3)$ be unipotent and irreducible, and assume that
	$\Gamma$ is a non-empty fence. 
	Then $\Gamma \subseteq \AA^3 \aquot \bu$ 
	is the largest closed subscheme fixed by $\Cent(\bu)_u$.
\end{prop}

\begin{proof}
	By Proposition~\ref{properties.prop} and Theorem~\ref{main.thm2},
	we get $p(\Cent(\bu)_u) = \Iner(\CC^2, \Gamma)_u$, where 
	$p \colon \Cent(\bu) \to \Aut(\CC^2, \Gamma)$ is the canonical morphism.
	Thus $\Gamma$ is fixed by the action of $\Cent(\bu)_u$. Let
	$X \subseteq \CC^2$ be a closed subscheme that is fixed under 
	$\Cent(\bu)_u$ and assume that $X$ contains $\Gamma$. Moreover, let
	$I(X) \subseteq \OO(\CC^2)$ be the ideal of $X$ and let
	$\Gamma = \div(a)$. By Proposition~\ref{genAMS.prop} there exist
	coordinates $(z, w)$ of $\CC^2$ such that $a \in \CC[z]$.
	Let $\sigma \in \Iner(\CC^2, \Gamma)_u$ be given by 
	$\sigma(z, w) = (z, w + a)$.
	By assumption, we get $a = \sigma^\ast(w)-w \in I(X)$. But this implies
	that $X$ is a closed subscheme of $\Gamma$ and hence $X = \Gamma$.
\end{proof}

\begin{prop}
	\label{char.prop}
	Let $\bid \neq \bu \in \Aut(\AA^3)$ be unipotent, not a translation, and let 
	$\bu = d \cdot \bu_0$ be a standard decomposition. Then the subgroup 
	$\OO(\AA^3)^{\bu_0} \cdot \bu_0$ of $\Cent(\bu)$ is characteristic, i.e.
	$\OO(\AA^3)^{\bu_0} \cdot \bu_0$ is invariant under all abstract group
	automorphisms of $\Cent(\bu)$.
\end{prop}

\begin{lem}
	\label{DiagonalAction.lem}
	Let $T$ be a torus acting on $\CC^2$. Assume that there exist coordinates
	$(z, w)$ of $\CC^2$ such that $z$ is a semi-invariant. Then there exists
	$r \in \CC[z]$ such that the action of $T$
	is diagonal with respect to the coordinate system $(z, w+r)$.
\end{lem}

\begin{proof}[Proof of Lemma~\ref{DiagonalAction.lem}]
	By assumption, $\pr \colon \AA^2 \to \AAone$, $(z, w) \to z$ 
	is $T$-equivariant  with respect to a suitable $T$-action on $\AAone$. Due to 
	\cite[Proposition~1]{KrKu1996Equivariant-affine}, every lift
	of a $T$-action on $\AAone$ to $\AA^2$ (with respect to $\pr$) 
	is equivalent to a trivial lift (with respect to $\pr$). Thus, there
	exists $r \in \CC[z]$ such that $w + r$ is a semi-invariant with respect
	to the action of $T$. This finishes the proof.
\end{proof}

\begin{proof}[Proof of Proposition~\ref{char.prop}]
Let $R \subseteq \Cent(\bu)$ be an algebraic subgroup such that
$\Cent(\bu) = R \ltimes N$ (see Theorem~\ref{main.thm1}).
Let $T = R^0$ be the connected component of the neutral element in $R$.
By Corollary~\ref{mainthm2.cor}, it is a torus (possibly $\dim T = 0$).
It follows from Theorem~\ref{main.thm1} that 
$T \ltimes N \subseteq \Cent(\bu)$ is a subgroup of finite index.
Moreover, $T \ltimes N$ has no proper subgroup of finite index, as this group
is generated by groups that have no proper subgroup of finite index. 
This implies that $T \ltimes N$ is a characteristic subgroup of $\Cent(\bu)$. 
Let $G = T \ltimes N$
and let $H = \OO(\AA^3)^{\bu_0} \cdot \bu_0 \subseteq N$. 
It is now enough to prove, that $H$ is characteristic in $G$.
We divide the proof now in two cases.

\textbf{$\Gamma$ is not a fence:}
	If $\dim T = 0$, then we have $H = G$. So assume $\dim T > 0$.
	There exist coordinates $(v_1, v_2)$ of $\CC^3 \aquot \bu$
	such that the action of $T$ on $\CC^3 \aquot \bu$ is diagonal
	with respect to $(v_1, v_2)$ (see \cite{Ka1979Automorphism-group}). 
	Let $\rho_1$ and $\rho_2$ be the
	characters of $T$ such that $\bt^\ast(v_i) = \rho_i(\bt) v_i$ 
	for all $\bt \in T$. Let $\bt \circ f \cdot \bu_0 \in \Cent_G(G^{(1)})$, where
	$G^{(1)} = [G, G]$ denotes the first derived group.
	A calculation shows for $i =1,2$ and for all $k \geq 0$
	\[
		\bid = [\bt \circ f \cdot \bu_0, [\bt^{-1}, v_i^k \cdot \bu_0]] 
		= (1-\mu(\bt^{-1})\rho_i(\bt^{-1})^k)(1-\mu(\bt)\rho_i(\bt)^k)v_i^k 
		\cdot \bu_0  
	\]
	where $\mu \colon T \to \CC^\ast$ is the character defined
	by $\mu(\bt) d = \bt^{\ast}(d)$ (see Remark~\ref{ActionByConj.rem}).
	Thus $\rho_i(\bt) = 1$ for $i =1,2$.
	As the action of $T$ on $\CC^3 \aquot \bu$ 
	is faithful, it follows that $\bt = 1$. Hence $H = \Cent_G(G^{(1)})$.

\textbf{$\Gamma$ is a fence:}
	Let $(z, P)$ be a coordinate system of $\AA^2 = \AA^3 \aquot \bu$
	such that $\Gamma \subseteq \AA^2$ is given by the standard embedding
	$F \times \AAone \subseteq \AA^2$ for some 0-dimensional 
	closed subscheme $F \subseteq \AAone$ and
	$\bu_0 = \Exp(\Delta_P)$ (see subsec.~\ref{The_second_subgroup.subsec}).
	As the torus $T$ preserves $\Gamma = V(a) \subseteq \CC^2$ 
	invariant and since $a \in \CC[z]$, there exists $q \in \CC$, such
	that $z+q$ is a semi-invariant for the action of $T$. By replacing 
	$z+q$ with $z$, we can assume that $z$ is a semi-invariant.
	Moreover, by replacing $P$ with a suitable $P+r$ for some $r \in \CC[z]$
	we can assume that the action of $T$ with respect to $(z, P)$ is diagonal 
	(see Lemma~\ref{DiagonalAction.lem}).
	Moreover, we denote by $\e$ an admissible complement to $\bu$.

	Assume first $\dim T = 0$. Let
	$h \cdot \e \circ f \cdot \bu_0 \in \Cent_G(G^{(1)})$. 
	A calculation shows that 
	\[
		\bid = 
		[h \cdot \e \circ f \cdot \bu_0, [P^2 \cdot \bu_0, \e]] = 
		-2h(a_0)^2 \cdot \bu_0 
	\] 
	where $a_0 \in \CC[z]$ such that $\Gamma_0 = \div(a_0)$.
	Hence $h = 0$ and therefore $H = \Cent_G(G^{(1)})$.
	
	Assume now, $\dim T > 0$. Let $A = \CC[z]$ and let $A \ltimes A[P]$ be the 
	semi-direct product defined by
	\begin{equation}
		\label{prod.eq}
		(h, f) \cdot (\bar{h}, \bar{f}) = (h + \bar{h}, f(P - \bar{h}a_0) + \bar{f})
	\end{equation}
	From Proposition~\ref{properties.prop} it follows that
	\[
		A \ltimes A[P] \stackrel{\sim}{\longrightarrow} N(\bu) \, , \quad
		(h, f) \longmapsto h \cdot \e \circ f \cdot \bu_0 
	\]
	is an isomorphism of groups. Under this isomorphism
	the subgroup $A[P]$ is sent onto $H$.
	It follows from Lemma~\ref{pres.lem} that $H = \Cent_G G^{(2)}$.
\end{proof}

\begin{lem}
	\label{pres.lem}
	Let $A \ltimes A[P]$ be defined as in \eqref{prod.eq} and let 
	$G = T \ltimes (A \ltimes A[P])$ where $T$ is a
	torus with $\dim T > 0$. Assume that 
	$A[P] \subseteq G$ is a normal subgroup, that
	the action of $T$ by conjugation on $A[P]$ is given by 
	$\lambda^{-1} f \lambda =  
	\mu(\lambda)f(\rho_1(\lambda)z, \rho_2(\lambda)P)$ 
	for some characters $\mu$, $\rho_1$, $\rho_2$ and that 
	$\ker \rho_1 \cap \ker \rho_2 = \{ 1 \}$.
	Furthermore we assume that the action
	of $T$ by conjugation on the quotient $G / A[P]$ 
	is non-trivial and the product in $G$ satisfies
	$(\lambda, 0, 0) \cdot (1, h, f) = (\lambda, h, f)$.
	Then $A[P] = \Cent_G G^{(2)}$.
\end{lem}

\begin{proof}[Proof of Lemma~\ref{pres.lem}]
	As the action by conjuagtion of $T$ on $G / A[P]$ is non-trivial, it follows
	that the first derived subgroup $G^{(1)}$ is not contained in 
	$A[P]$. As $T$ is abelian it follows
	that $G^{(1)} \subseteq A \ltimes A[P]$ and as $A$ is abelian
	we conclude $G^{(2)} \subseteq A[P]$.
	Thus there exists $(1, h_0, f_0) \in G^{(1)}$ with $h_0 \neq 0$. 
	As $A[P]$ is abelian, we get $A[P] \subseteq \Cent_G G^{(2)}$. 
	Now, we prove $\Cent_G G^{(2)} \subseteq A[P]$. We have
	\[
		(1, 0, q - q(P + h_0a_0)) = [(1, 0, q), (1, h_0, f_0)] \in G^{(2)}
		\quad \textrm{for all $(1, 0, q) \in G^{(1)}$} \, .
	\]
	Moreover, $(1, 0, z^i P^j) \in G^{(1)}$
	for all $(i, j) \in \NN_0^2$ such that $\mu \rho_1^i \rho_2^j$ is not the trivial
	character, as we have $(1, 0, z^i P^j) = [(\lambda, 0, 0), (1, 0, z^i P^j)]$ 
	for some well chosen $\lambda \in T$. For all $j \geq 0$, 
	the character $\mu \rho_1^i \rho_2^j$ is non-trivial,
	provided that $i$ is large enough. For all $(1, 0, f) \in A[P]$ we have
	\begin{eqnarray*}
		\Cent_G(1, 0, f) & = &
		\{ \, (\bar{\lambda}, \bar{h}, \bar{f}) \in G \ | \ 
		      (\bar{\lambda}, \bar{h}, \bar{f})^{-1} \cdot (1, 0, f) \cdot 
		      (\bar{\lambda}, \bar{h}, \bar{f}) = (1, 0, f)  \, \} \\ 
		&=& \{ \, (\bar{\lambda}, \bar{h}, \bar{f}) \in G \ | \ f =
		(\bar{\lambda}^{-1}f\bar{\lambda})(P - \bar{h}a_0) \, \}
	\end{eqnarray*}
	Let $(\bar{\lambda}, \bar{h}, \bar{f}) \in \Cent_G G^{(2)}$. 
	Since 
	$(\bar{\lambda}, \bar{h}, \bar{f}) \in \Cent_G(1, 0, -z^i h_0 a_0)$ 
	for $i$ sufficiently large, we get $\bar{\lambda} \in \ker \rho_1$.
	Moreover, 
	$(\bar{\lambda}, \bar{h}, \bar{f}) \in 
	\Cent_G(1, 0, -z^i h_0 a_0 (2P + h_0 a_0))$
	for sufficiently large $i$. This implies
	$\rho_2(\bar{\lambda}) = \mu(\bar{\lambda})^{-1}$ and 
	$\bar{h} = ((\mu(\bar{\lambda})-1)/2)h_0$. Since 
	$(\bar{\lambda}, \bar{h}, \bar{f}) \in 
	\Cent_G(1, 0, -z^i h_0 a_0 (3P^2+3Ph_0 a_0+(h_0 a_0)^2))$ it follows that
	$\rho_2(\bar{\lambda})^2 = \mu(\bar{\lambda})^{-1}$.
	Therefore $\bar{\lambda} \in \ker \rho_2$, $\bar{h} = 0$. Hence we have
	$(\bar{\lambda}, \bar{h}, \bar{f}) = (1, 0, \bar{f}) \in A[P]$ and this proves 
	$\Cent_G G^{(2)} \subseteq A[P]$.
\end{proof}

\section{Acknowledgements}

I would like to thank \name{Hanspeter Kraft} for many fruitful discussions.
Many thanks go also to the referee for his helpful comments.

\vskip 0.5cm

\bibliographystyle{amsalpha}

\providecommand{\bysame}{\leavevmode\hbox to3em{\hrulefill}\thinspace}
\providecommand{\MR}{\relax\ifhmode\unskip\space\fi MR }
\providecommand{\MRhref}[2]{%
  \href{http://www.ams.org/mathscinet-getitem?mr=#1}{#2}
}
\providecommand{\href}[2]{#2}

\vskip 0.5cm

\end{document}